\documentclass[12pt]{amsart}
\usepackage{amsmath, amsthm, graphicx, amssymb,verbatim,pst-all,color}
\usepackage[all]{xy}
\usepackage[margin=1.0in]{geometry}
\usepackage{mathptmx}
\ifx\JPicScale\undefined\def\JPicScale{1.0}\fi
\unitlength \JPicScale mm

{\theoremstyle{plain}
   \newtheorem*{maintheorem}{Theorem}
   \newtheorem{theorem}{Theorem}[section]
   \newtheorem{proposition}[theorem]{Proposition}     
   \newtheorem{lemma}[theorem]{Lemma}
   
   \newtheorem{corollary}[theorem]{Corollary}

}
{\theoremstyle{definition}
   \newtheorem*{ack}{Acknowledgements}

   \newtheorem{example}[theorem]{Example}
   \newtheorem{definition}[theorem]{Definition}
   \newtheorem{remark}[theorem]{Remark}

}

\newcommand{\ChowQ}{/\hspace{-1.2mm}/_{Ch}}

\newcommand{\M}{\overline{M}}
\newcommand{\SL}{\operatorname{SL}}
\newcommand{\PGL}{\operatorname{PGL}}
\newcommand{\Chow}{\operatorname{Chow}}

\newcommand{\Spec}{\operatorname{Spec}}

\newcommand{\Bl}{\operatorname{Bl}}

\newcommand{\Ga}{\mathbb{G}_a}
\newcommand{\Gm}{\mathbb{G}_m}
\newcommand{\PP}{\mathbb{P}}

\numberwithin{theorem}{section}

\begin{document}
\title{Modular interpretation of a non-reductive Chow quotient}
\author{Patricio Gallardo and Noah Giansiracusa}
\maketitle


\begin{abstract}
The space of $n$ distinct points and a disjoint parameterized hyperplane in projective $d$-space up to projectivity---equivalently, configurations of $n$ distinct points in affine $d$-space up to translation and homothety---has a beautiful compactification introduced by Chen-Gibney-Krashen.  This variety, constructed inductively using the apparatus of Fulton-MacPherson configuration spaces, is a parameter space of certain pointed rational varieties whose dual intersection complex is a rooted tree.  This generalizes $\M_{0,n}$ and shares many properties with it.  In this paper, we prove that the normalization of the Chow quotient of $(\PP^d)^n$ by the diagonal action of the subgroup of projectivities fixing a hyperplane, pointwise, is isomorphic to this Chen-Gibney-Krashen space $T_{d,n}$.  This is a non-reductive analogue of Kapranov's famous quotient construction of $\M_{0,n}$, and indeed as a special case we show that $\M_{0,n}$ is the Chow quotient of $(\PP^1)^{n-1}$ by an action of $\Gm\rtimes\Ga$.
\end{abstract} 




\section{Introduction}

\subsection{Main result}
The spaces $T_{d,n}$ introduced by Chen-Gibney-Krashen parameterize stable rooted trees of pointed $\PP^d$s (cf., \S\ref{section:CGK}); they compactify the space of $n$ distinct points and a disjoint parameterized hyperplane in $\PP^d$ up to projective automorphisms preserving the parameterization, or equivalently, configurations of $n$ distinct points in $\mathbb{A}^d$ up to translation and homothety \cite{Chen-Gibney-Krashen}.  Since $T_{1,n} \cong \M_{0,n+1}$, a natural question is whether Kapranov's Chow quotient construction $\M_{0,n} \cong (\PP^1)^n\ChowQ \SL_2$ \cite{Kapranov-chow} extends to the Chen-Gibney-Krashen moduli spaces.  We prove that, up to normalization, this is indeed the case.

\begin{maintheorem}\label{thm:main}
For any $d\ge 1$ and $n \ge 2$, $T_{d,n}$ is isomorphic to the normalization of the Chow quotient $(\PP^d)^n\ChowQ G$, where $G\subseteq\SL_{d+1}$ is the non-reductive subgroup fixing a hyperplane pointwise.
\end{maintheorem}

Put another way, this non-reductive Chow quotient normalization admits a surprisingly elegant modular interpretation.  When $d=1$ we avoid the normalization step and hence obtain a novel construction of the moduli space of stable rational curves. 
\begin{maintheorem}
For any $n \ge 3$, we have $\M_{0,n} \cong (\PP^1)^{n-1}\ChowQ (\Gm\rtimes \Ga)$.  
\end{maintheorem}
Since $(\PP^1)^{n-1}\ChowQ\Gm$ is a toric variety, by general results of \cite{ChowToric}, this gives one more instance of the philosophy in \cite{Doran-Giansiracusa} that $\M_{0,n}$ is one additive group $\Ga$ away from being toric.

\subsection{Background and motivation}
As an ``elementary'' example of his Geometric Invariant Theory, Mumford provided a family of compactifications of the space of $n$ distinct points in projective space up to projectivity \cite[Chapter 3]{GIT}.  Kapranov later studied a new kind of quotient, the Chow quotient, defined as the closure in the relevant Chow variety of the space of algebraic cycles associated to generic orbit closures \cite{Kapranov-chow}.  His main example $(\PP^d)^n\ChowQ\SL_{d+1}$ provides a finer compactification than Mumford's of point configurations in $\PP^d$, and in the case $d=1$ this quotient is isomorphic to the ubiquitous Grothendieck-Knudsen compactification $\M_{0,n}$ \cite[Theorem 4.1.8]{Kapranov-chow} (see also \cite{Giansiracusa-Gillam}).  Thus Kapranov's spaces $(\PP^d)^n\ChowQ\SL_{d+1}$ are seen as higher-dimensional generalizations of $\M_{0,n}$; they appear in a variety of settings (e.g., \cite{Keel-Tevelev-chow, Hacking-Keel-Tevelev-hyperplane, Alexeev-weighted, Lafforgue}).

Another family of higher-dimensional generalizations of $\M_{0,n}$, the aforementioned parameter spaces $T_{d,n}$, were introduced by Chen, Gibney, and Krashen \cite{Chen-Gibney-Krashen}.  These too appear in a variety of settings (e.g., \cite{Westerland,Kim-Sato,Manin-Marcolli}), and like Kapranov's spaces they recover Grothendieck-Knusden when $d=1$.  Strikingly, however, many beautiful properties of $\M_{0,n}$ carry over to the $d>1$ case: $T_{d,n}$ is a smooth projective variety with an explicit functorial description, its modular boundary divisor has simple normal crossings and a recursive description, and its closed points parameterize stable rooted trees of pointed projective spaces, a direct generalization of stable trees of pointed projective lines.  The construction of the CGK spaces $T_{d,n}$ in \cite{Chen-Gibney-Krashen} is a subtle induction based on the Fulton-MacPherson configuration spaces \cite{Fulton-MacPherson}.  The original construction of $\M_{0,n}$ was also a subtle induction \cite{Knudsen} (see also \cite{Keel-thesis}), and it was Kapranov's alternative ``global'' constructions \cite{Kapranov-chow, Kapranov-veronese} that illuminated much of the geometry of $\M_{0,n}$ that has been explored in the past 20 years (e.g., \cite{Doran-Giansiracusa, Doran-Giansiracusa-Jensen, Castravet-Tevelev-mds,Gibney-Maclagan-quotients, Bruno-Mella, Pixton, Hu-degen}). 

In this paper, we introduce a construction of $T_{d,n}$ as a non-reductive Chow quotient of $(\PP^d)^n$.  Indeed, for the standard action of $\SL_{d+1}$ on $\PP^d$, the subgroup fixing a hyperplane pointwise is a solvable, non-reductive group $G \cong \Gm\rtimes \Ga^d$, and we show that $T_{d,n}$ is isomorphic to the normalization of $(\PP^d)^n\ChowQ G$.  It follows in particular that, up to normalization, the Chow quotient of $(\PP^d)^n$ by the subgroup $G$ is smooth with simple normal crossings boundary, as opposed to the Chow quotient by $\SL_{d+1}$.  Our proof follows closely the reasoning in \cite{Giansiracusa-Gillam,Giansiracusa-rnc}, where the second author  and Gillam reprove and generalize Kapranov's isomorphism $\M_{0,n} \cong (\PP^1)^n\ChowQ \SL_2$ by describing explicitly the union of orbit closures corresponding to points in the boundary. Note that while Mumford's GIT setup crucially requires the reductivity hypothesis and avoiding this is a rather delicate affair \cite{Doran-Kirwan, Kirwan}, Kapranov's Chow quotient setup on the other hand makes no such assumption, since it is not based on rings of invariants.

Since we rely extensively on results and constructions from Kapranov's paper \cite{Kapranov-chow}, we work over the complex numbers $\mathbb{C}$.  This hypothesis could likely be weakened if necessary by using a more modern approach to Chow varieties (e.g., \cite[Chapter 1]{Kollar-chow}), as was done in \cite{Giansiracusa-Gillam}.

\subsection{Navigating the proof}

The first step is to recognize that the parameter space $T_{d,n}$ is birational to the Chow quotient $(\PP^d)^n\ChowQ G$.  This follows from the observation that the open dense stratum in $T_{d,n}$ parameterizes $G$-orbits of configurations $p=(p_1,\ldots,p_n)\in(\PP^d)^n$ of $n$ distinct points in $\PP^d$ not lying on the parameterized hyperplane; if these points are in general position then the orbit closures $\overline{Gp}$ all have the same homology class and the map $Gp \mapsto \overline{Gp}$ explicitly identifies an open dense subset of $T_{d,n}$ with one in the Chow quotient.  See \S\ref{sec:birat} for details.

The next step, which lies at the technical heart of the paper and is covered in \S\ref{sec:regular}, is to extend the birational map $T_{d,n} \dashrightarrow (\PP^d)^n\ChowQ G$ to a regular morphism.  To do this, we use a criterion developed in \cite{Giansiracusa-Gillam}: it suffices to associate an algebraic cycle, of the same homology class as the generic orbit closure, to each boundary point of $T_{d,n}$, and then to show that this association is compatible with 1-parameter families in an appropriate sense.  To a stable tree of pointed $\PP^d$s, we associate the following cycle: for each irreducible component there is a configuration of not necessarily distinct points obtained by contracting down and projecting up to this component $\PP^d$ (\S\ref{sec:compconf}), and we consider the cycle given by the union of orbit closures over all components.  To show compatibility with 1-parameter families we reduce to the case of maximally degenerate pointed trees, i.e., the highest codimension boundary strata in $T_{d,n}$, and then use an argument based on continuity.

Finally, in \S\ref{sec:iso} we invoke Zariski's Main Theorem, so that to prove $T_{d,n}$ is isomorphic to the normalization of the Chow quotient it suffices to show that the morphism $T_{d,n} \rightarrow (\PP^d)^n\ChowQ G$ just constructed is bijective.  Surjectivity follows from continuity, and injectivity on the open stratum of $T_{d,n}$ is straightforward.  To prove injectivity on the boundary divisor we prove a certain compatibility between our morphism and the recursive structure of the boundary of $T_{d,n}$; this allows us in essence to reduce to the case $n=2$ where injectivity is once again straightforward.  

When $d=1$, by using the identification $T_{d,n} \cong \M_{0,n+1}$ and an idea from \cite{Giansiracusa-Gillam} we are able to avoid the use of Zariski's Main Theorem and reduce the question of whether our morphism is an isomorphism to the case $n=3$, namely $\PP^1 \cong T_{1,3} \rightarrow (\PP^1)^3\ChowQ(\Gm\rtimes\Ga)$, where we can directly argue that it is indeed an isomorphism.

\begin{remark}
We have so far been unable to determine whether the normalization statement in our main theorem is necessary when $d >1$ or whether it is merely an artifact of this use of Zariski's Main Theorem.  The main challenge is that the deformation theory of both $T_{d,n}$ and of Chow varieties are not well understood so it would be difficult to show that our morphism separates tangents vectors in addition to separating points.
\end{remark}

\begin{ack}
We thank Valery Alexeev, Angela Gibney, and Danny Krashen for helpful discussions.  The first author was supported in part by the NSF grant DMS-1344994 of the RTG in Algebra, Algebraic Geometry, and Number Theory, at the University of Georgia.  The second author was partially supported by the NSF grant DMS-1204440.
\end{ack}


\section{Chen-Gibney-Krashen parameter spaces}\label{section:CGK}

In this section we summarize the relevant aspects of the CGK parameter spaces $T_{d,n}$ from \cite{Chen-Gibney-Krashen} that we will need throughout this paper.  We begin by recalling the closed points of these spaces and introduce some terminology we shall use in discussing the rational varieties they parameterize and the trees associated to them.  We then list the salient properties that will be used later in the paper.

\subsection{Closed points}\label{cpts}

Recall that the closed points of $\M_{0,n}$ are in bijection with nodal unions of $\PP^1$s whose dual graph is a tree, such that the $n$ marked points are distinct and non-singular and each component has at least three special points (nodes or marked points); these are considered up to isomorphism of pointed curves.  We can view the points of $\M_{0,n+1}$ as the same type of object, but now one vertex of the tree is distinguished, namely, the vertex whose associated component carries the $(n+1)$st marked point (which we can, if desired, view as a hyperplane in the corresponding $\PP^1$).  Such trees are called \emph{rooted}.  This is the perspective of the isomorphism $T_{1,n} \cong \M_{0,n+1}$ \cite[Proposition 3.4.3]{Chen-Gibney-Krashen}.

By \cite[Theorem 3.4.4]{Chen-Gibney-Krashen}, the closed points of $T_{d,n}$ parametrize ``$n$-pointed stable rooted trees of $d$-dimensional projective spaces.''  These are $n$-pointed rational varieties $X$, possibly reducible but with simple normal crossings, whose irreducible components $X_i \subseteq X$ are each equipped with a closed immersion $\PP^{d-1} \hookrightarrow X_i$ and an isomorphism $X_i \cong \Bl_{k_i} \PP^d$ to the blow-up of $\PP^d$ at a collection of $k_i$ points, such that the conditions listed below hold.  Before stating these conditions, we introduce some terminology.  For each irreducible component $X_i \subseteq X$, we write $\overline{X}_i$ for the image of the composition with the blow-down morphism: \[X_i \cong \Bl_{k_i} \PP^d \rightarrow \PP^d =: \overline{X_i}.\]  By the \emph{blown-up points of $\overline{X_i}$} we mean the points in $\overline{X_i}$ that are the images of the $k_i$ exceptional divisors.  We refer to these points and the image of the marked points under the above composition $X_i \rightarrow \overline{X_i}$ as the \emph{special points} of $\overline{X_i}$.  Finally, we denote the image of the map $\PP^{d-1} \hookrightarrow X_i$ by $H_i$.  Here are now the conditions that characterize when this data forms a closed point of $T_{d,n}$:
\begin{enumerate}
\item the dual intersection complex of $X$ is a tree graph;
\item any nonempty intersection of two components, $X_i \cap X_j$, is an exceptional divisor in one component, say $X_i \cong \Bl_{k_i}\PP^d$, and in the other it is $H_j \subseteq X_j$;
\item there is a unique component $X_0$ such that the number of blown-up points $k_0$ equals the number of distinct components meeting $X_0$; for all other components $X_i$, the number $k_i$ is the number of distinct components meeting $X_i$ minus one;
\item each $\overline{X_i}$ has at least two special points;
\item the image of $H_i \subseteq X_i$ under the map $X_i \rightarrow \overline{X_i}$ is a hyperplane disjoint from the special points.
\end{enumerate}
 We shall abuse notation slightly and refer to $H_i \subset X_i$ as a hyperplane, whereas technically its image under the isomorphism $X_i \cong \Bl_{k_i}\PP^d$ is the strict transform of a hyperplane.    Note that the special component $X_0 \subseteq X$ determines a root of the dual graph tree, so we call it the \emph{root component}; we call the corresponding hyperplane $H_0 \subseteq X_0$ the \emph{root hyperplane}.  This determines a partial order $\le$ on the vertices of the dual graph tree, with the root being the smallest element.
\begin{definition}\label{family}
If vertices $v,v'$ in a rooted tree satisfy $v < v'$, then we say that $v'$ is a \emph{descendent} of $v$ and $v$ is an \emph{ancestor} of $v'$.  If there is no vertex $w$ satisfying $v < w < v'$ then we say that $v'$ is a \emph{daughter} of $v$ and $v$ is the \emph{parent} of $v'$.
\end{definition}
Concretely, then, each irreducible component $X_v \subseteq X$ is isomorphic to $\Bl_{k_v}\PP^d$ where $k_v$ is the number of daughters of $v$; if $v$ is not the root vertex then the hyperplane $H_v \subseteq X_v$ is identified with an exceptional divisor in the parent component of $X_v$.  

We sometimes denote the data of a stable $n$-pointed tree of $\PP^d$s by
\[\underline{X} = (H_{v_0} \hookrightarrow X = \cup_{v\in V}X_v,q_1,\ldots,q_n) \in T_{d,n},\] to indicate that $X$ is the rational variety itself, $V$ is the vertex set of the dual graph, $v_0\in V$ is the root vertex, $H_{v_0}$ is the root hyperplane, and $q_i$ are the marked point.

These $n$-pointed rational trees $\underline{X}$ are considered up to isomorphisms compatible with the $n$ marked points and the embedding of the root hyperplane.  In particular, an automorphism must fix each marked point and each point of the root hyperplane.  Conditions (1-3) imply that an automorphism must send each component $X_i$ to itself and the restriction to $X_i$ must be induced by an automorphism of the blow-down $\overline{X_i}$.  We claim conditions (4-5), the ``stability conditions,'' are then equivalent to the statement that there are no non-trivial automorphisms.  By induction it suffices to check this for the root component, so the preceding claim is equivalent to the claim that there are no non-trivial automorphisms of $\PP^d$ fixing two points and fixing a disjoint hyperplane pointwise, whereas there is a non-trivial automorphism fixing one point and a disjoint hyperplane pointwise.  Note that fixing a hyperplane $H \subset \PP^d$ pointwise is equivalent to fixing $d$ points which span it together with an additional generic point in $H$; by choosing explicit coordinates (cf., Lemma \ref{lem:ptfix}) one sees that there is a $\Gm$-stabilizer for the configuration of these $d+1$ points together with any additional point outside $H$, so having only one special point is unstable, whereas adding any distinct second point outside $H$ kills this stabilizer and results in a stable configuration.

The closed points of $T_{d,n}$ are stratified by the corresponding dual graphs in the usual way, with the dense open stratum corresponding to a tree consisting of a single vertex; the associated varieties are simply $\PP^d$ with $n$ distinct marked points and a disjoint hyperplane $\PP^{d-1} \subseteq \PP^d$.

\begin{example}\label{ex1}
We illustrate here a closed point in the boundary of $T_{2,4}$. The root vertex $v_0$ has daughters $v_1$ and $v_2$.  The components $X_{v_1} \cong \PP^2 \cong X_{v_2}$ each have two marked points $q_i$ while the root component $X_{v_0} \cong \Bl_2\PP^2$ has no marked points but $\overline{X_{v_0}} =\PP^2$ has two special points.
\begin{figure}[h!]
\includegraphics[scale=0.5]{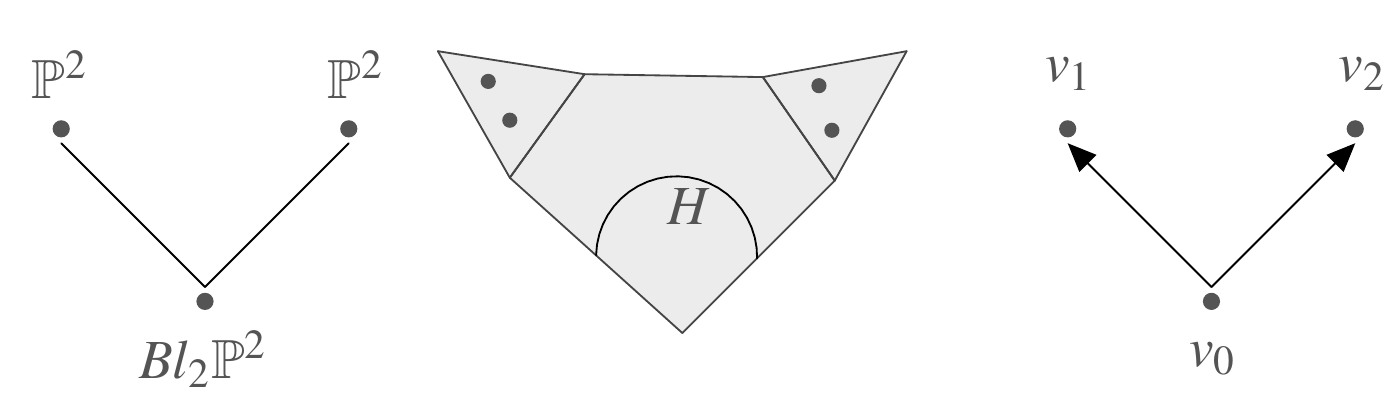}
\caption{A boundary point of $T_{2,4}$.}
\end{figure}
\end{example}

By varying the location of the two special points in each component we get a 3-dimensional irreducible component of a boundary stratum in the $5$-dimensional moduli space $T_{2,4}$; this component is isomorphic to $T_{2,2} \times T_{2,2} \times T_{2,2}$ (see below).

\subsection{Basic properties}\label{section:props}
The following results will be useful in what follows.
\begin{itemize}
\item $T_{d,n}$ is a smooth, projective, rational variety of dimension $dn - d - 1$ \cite[Corollary 3.4.2]{Chen-Gibney-Krashen};
\item $T_{1,3} \cong \PP^1$, $T_{d,2} \cong \PP^{d-1}$, and $T_{1,n} \cong \M_{0,n+1}$ \cite[Propositions 3.4.1 and 3.4.3]{Chen-Gibney-Krashen};
\item the cycle class map $A^*(T_{d,n}) \rightarrow H^{2*}(T_{d,n},\mathbb{Z})$ is an isomorphism \cite[Corollary 7.3.4]{Chen-Gibney-Krashen};
\item the boundary divisor in $T_{d,n}$, parameterizing stable trees with more than one vertex, has simple normal crossings; its components are isomorphic to products $T_{d,n-i+1}\times T_{d,i}$ and parameterize stable trees where a collection of $i$ points have collided and are placed on a non-root component \cite[\S1.1]{Chen-Gibney-Krashen} (see Figure 2 for an example);
\item there is a ``universal family'' $T^+_{d,n} \rightarrow T_{d,n}$ of stable rooted trees of $d$-dimensional projective spaces (cf., the proof of \cite[Theorem 3.4.4]{Chen-Gibney-Krashen}).
\end{itemize}


\begin{remark}\label{rem:unifam}
The quotes in this last item require explanation.  The morphism $T^+_{d,n} \rightarrow T_{d,n}$ is flat and proper, with fibers given by the rational varieties $X$ described in \S\ref{cpts}.  Thus, given any morphism $Z \rightarrow T_{d,n}$ one can pull back this universal family to get a flat proper family of stable rooted trees of $d$-dimensional projective spaces over $Z$.  However, it is not currently known whether all such families over $Z$ are the pull-back of $T^+_{d,n}$ along a (unique) morphism $Z \rightarrow T_{d,n}$.
\end{remark}

\begin{remark}
The isomorphism $T_{d,2} \cong \PP^{d-1}$ admits a nice elementary geometric interpretation.  Given a pair of distinct points $p,q\in \PP^d$ and a disjoint hyperplane $H \hookrightarrow \PP^d$, the group of projectivities fixing $p$ and each point of $H$ is a copy of $\Gm$ which acts on the line $\overline{pq} \cong \PP^1$ by the usual scaling action with $p$ identified as the origin in this line and the intersection with $H$ as the point at infinity.  It follows that the points of $T_{d,2}$ are in natural bijection with the lines in $\PP^d$ through $p$, which of course is $\PP^{d-1}$.
\end{remark}

\begin{figure}[h!]\label{fig:strata}
\includegraphics[scale=0.5]{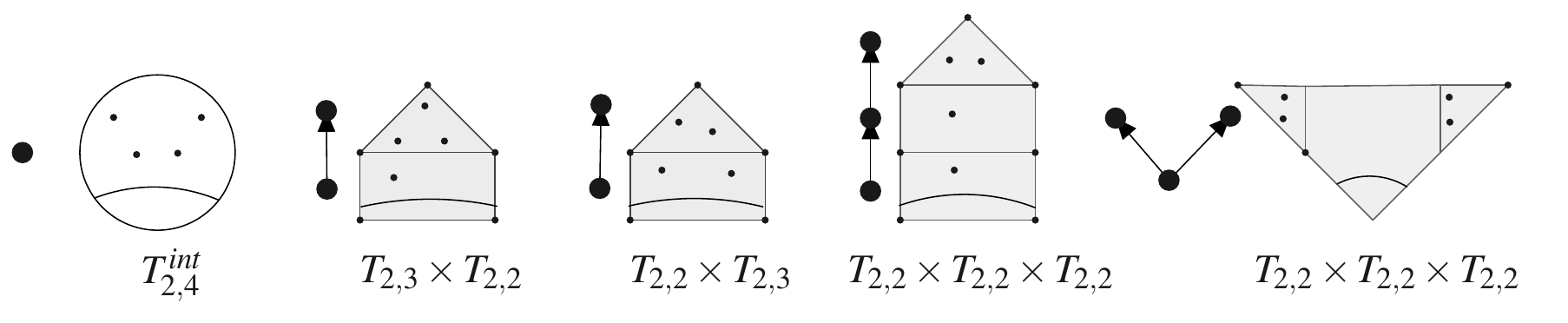}
\caption{The stratification for $T_{2,4}$.}
\end{figure}


\section{A birational map from CGK to Chow}\label{sec:birat}

In this section we show that there is a natural birational map $\rho : T_{d,n} \dashrightarrow (\PP^d)^n\ChowQ G$ by identifying common open subsets of these two irreducible varieties.

\subsection{The group and its action}
Fix integers $n \ge 2$ and $d \ge 1$, and consider the natural action of $\SL_{d+1}$ on $\mathbb{P}^d$.  Let $G\subseteq \SL_{d+1}$ be the subgroup fixing, pointwise, the hyperplane $H \subseteq \PP^d$ defined by the vanishing of the first coordinate.  There is a diagonal action of $G$ on $(\PP^d)^n$, the space of $n$ points in $\PP^d$.  We use the following elementary linear algebraic observation throughout:

\begin{lemma}\label{lem:ptfix}
An element of $\SL_{d+1}$ fixes $H$ pointwise if and only if it fixes $d+1$ general points of $H$.  These can be taken to be the $d$ coordinate points $e_2,\ldots,e_{d+1}$ and the point $e_2+\cdots +e_{d+1}$. 
\end{lemma}

For instance, this yields an explicit description of the group $G$:

\begin{proposition}\label{prop:matrix}
The subgroup fixing a hyperplane pointwise is a semidirect product $G \cong \Gm\rtimes \Ga^d$ of the multiplicative group and $d$ copies of the additive group.  For the hyerplane $H$ defined above, it consists of matrices of the form
\[\left(\begin{array}{cccc}t^{-d} &  0 & \cdots & 0 \\s_1 & t &  & 0 \\
\vdots &  & \ddots &  \\s_d & 0  & \cdots & t \end{array} 
\right)\] 
for $s_i \in \mathbb{C}$ and $t\in\mathbb{C}^*$.
\end{proposition}

\begin{proof}
The first claim follows from the second, and the second follows by direct computation from Lemma \ref{lem:ptfix}.
\end{proof}

\subsection{Quotient of the generic locus}
Let $U \subseteq (\PP^d)^n$ denote an arbitrarily small $G$-invariant open locus, corresponding to points in general position, on which $G$ acts freely.

\begin{proposition}\label{prop:openquot}
The quotient $U/G$ is naturally an open, dense subvariety of $T_{d,n}$.
\end{proposition}

\begin{proof}
A configuration \[p = (p_1,\ldots,p_n) \in U \subseteq (\PP^d)^n\] of generic points in $\PP^d$ necessarily consists of distinct points that are disjoint from the hyperplane $H\subseteq \PP^d$, so $p$ determines a point in the interior of the moduli space $T_{d,n}$, i.e., a stable rooted tree of pointed projective spaces whose graph has only one vertex.  Since this association is algebraic, we have a natural map $\psi : U \rightarrow T_{d,n}$.  Since the $G$-action fixes $H$ pointwise, all configurations in the orbit $Gp$ are isomorphic as stable rooted trees of pointed projective spaces and hence correspond to the same point of $T_{d,n}$; thus, $\psi$ is $G$-invariant and induces a morphism $\overline{\psi} : U/G \rightarrow T_{d,n}$.  

We first observe that $\overline{\psi}$ is injective.  Indeed, if two configurations $p,p'\in U$ yield isomorphic stable trees, then since they both correspond to points in $\PP^d$ with the same root hyperplane $H\subseteq \PP^d$, there must be a projective automorphism, hence matrix $g\in\SL_{d+1}$, sending $p$ to $p'$ and fixing the hyperplane $H$ pointwise; the latter condition implies that $g\in G$, so $p$ and $p'$ are in the same $G$-orbit.  Next, we note that the image of $\overline{\psi}$ is manifestly open and dense in the irreducible variety $T_{d,n}$.  Since we are working with varieties over $\mathbb{C}$, this implies that $\overline{\psi}$ is birational, hence by shrinking $U$ if necessary we can assume that $\overline{\psi}$ is an isomorphism onto its open, dense image.
\end{proof}

\subsection{Compactifying with the Chow quotient}\label{subsection:delta}
The orbit closures $\overline{Gp} \subseteq (\PP^d)^n$ for $p\in U$, are all of the same dimension $d+1$, and by shrinking $U$ if necessary we can assume they all have the same homology class \[\delta := [\overline{Gp}]\in H_{2d+2}((\PP^d)^n,\mathbb{Z}).\]  We then have an open immersion \[U/G \hookrightarrow \Chow((\PP^d)^n,\delta)\] into the Chow variety parameterizing effective algebraic cycles with the homology class $\delta$ of the generic orbit closure.  By definition, the Chow quotient $(\PP^d)^n\ChowQ G$ is the closure $\overline{U/G}$ in this embedding \cite[Definition 0.1.7]{Kapranov-chow}.  Note that, as opposed to Mumford's GIT quotients, these Chow quotients make no assumption that the group acting is reductive.

\begin{corollary}\label{cor:BiratCGK}
The Chow quotient $(\PP^d)^n\ChowQ G$ is an irreducible projective variety birational to the Chen-Gibney-Krashen moduli space $T_{d,n}$.
\end{corollary}

\begin{proof}
The Chow quotient is irreducible since $U/G$ is irreducible, and it is projective since the Chow variety is projective.  By Proposition \ref{prop:openquot} and the definition of the Chow quotient, we have open immersions
\[T_{d,n} \hookleftarrow U/G \hookrightarrow (\PP^d)^n\ChowQ G\]
with dense image, hence $T_{d,n}$ and $(\PP^d)^n\ChowQ G$ are birational.
\end{proof}

\begin{definition}
We shall refer to the birational map
\[\rho : T_{d,n} \dashrightarrow (\PP^d)^n\ChowQ G\] 
induced by Corollary \ref{cor:BiratCGK} as the \emph{CGK-to-Chow map}.
\end{definition}

The CGK-to-Chow map $\rho$ is the focus of the rest of this paper.


\section{The CGK-to-Chow map is a regular morphism}\label{sec:regular}

In this section we prove that the birational map $\rho : T_{d,n} \dashrightarrow (\PP^d)^n\ChowQ G$ constructed in the previous section is a regular morphism.  We do this by taking 1-parameter families in the parameter space from the interior to the boundary and defining the image of their limit to be the limit of their image.  We begin by studying the homology classes of various orbit closures and then describe a construction that associates to any closed point of $T_{d,n}$ a cycle of the appropriate homology class obtained as a sum of orbit closure cycles.  With this construction in place, we conclude by showing that generic orbit closure cycles degenerate in the Chow variety to these sums in a manner compatible with the way pointed, rooted $\PP^d$s degenerates to stable trees of such objects in $T_{d,n}$. 

\subsection{Homology classes of orbit closures}

We wish to compute the homology class \[[\overline{Gp}] \in H_{2(d+1)}((\PP^d)^n,\mathbb{Z})\] of the closure of the orbit of various configurations of points $p = (p_1,\ldots,p_n)$, when this orbit is full-dimensional.  For instance, for a generic configuration this is the class $\delta$ from \S\ref{subsection:delta} determining the component of the Chow variety into which our Chow quotient embeds.

\begin{proposition}\label{prop:fulldim}
The orbit $Gp$ has full dimension $d+1$ if and only if the support of the configuration $p \in (\PP^d)^n$ contains at least two distinct points not lying on the hyperplane $H\subseteq\PP^d$.
\end{proposition}

\begin{proof}
The orbit is full-dimensional if and only if the stabilizer is zero-dimensional.  By choosing, without loss of generality, one of the $p_i$ to be the coordinate point $(1 : 0 : \cdots : 0)$, one readily sees that the stabilizer of a configuration whose support has only one point outside $H$ is $\Gm$, and adding any additional point outside this point and $H$ results in a finite stabilizer.
\end{proof}

We focus now only on full-dimensional orbits $Gp$.  By the K\"unneth formula, \[H_*((\PP^d)^n,\mathbb{Z}) \cong \bigotimes_{i=1}^n H_*(\PP^d,\mathbb{Z}),\] so a basis for the relevant homology group is the collection of tensor products \[[\PP^{m_1}]\otimes\cdots\otimes[\PP^{m_n}],~\sum_{i=1}^n m_i = d+1,\] where each $\PP^{m_i} \subseteq \PP^d$ is a linear subspace.  

\begin{proposition}\label{prop:homclass}
The coefficient of $[\overline{Gp}]$ on the basis element $[\PP^{m_1}]\otimes\cdots\otimes[\PP^{m_n}]$ is either 0 or 1, and it is 1 if and only the following holds: for general linear subspaces $L_i \subseteq \PP^d$ of codimension $m_i$, there is a unique $g\in G$ such that $g\cdot p_i \in L_i$ for $1 \le i \le n$.
\end{proposition}

\begin{proof}
Kapranov proves the analogous result for $\PGL_{d+1}$ acting on $(\PP^d)^n$ en route to proving \cite[Proposition 2.1.7]{Kapranov-chow}.  His argument works verbatim in our setting simply by changing all instances of $\PGL_{d+1}$ to $G$ and all instances of $(d+1)^2-1$ to $d+1$.  Here is an outline.  

The coefficient of the cycle $[\overline{Gp}]$ on $[\PP^{m_1}]\otimes\cdots\otimes[\PP^{m_n}]$ is by definition the multi-degree of this cycle determined by the indices $m_i$, which means it is the intersection number of the subvariety $\overline{Gp} \subseteq (\PP^d)^n$ with the product of generic linear subspaces $L_i \subseteq \PP^d$ of codimension $m_i$.  This intersection number is finite, since $\dim G = \sum_{i=1}^n m_i$, and it coincides with the cardinality of the set \[\{g\in G~|~g\cdot p \in L_1 \times \cdots \times L_n\}.\]  Let $\widehat{L_i} \subseteq \mathbb{A}^{d+1}$ denote the affine cone over $L_i\subseteq \PP^d$ and consider the action of the algebra $M_{d+1}$ of all square size $d+1$ matrices on $\mathbb{A}^{d+1}$.  For $g\in M_{d+1}$, the condition $g\cdot \widehat{p}_i \in \widehat{L}_i$, where $\widehat{p} = (\widehat{p}_1,\ldots,\widehat{p_n}) \in (\mathbb{A}^{d+1})^n$ is any lift of $p \in (\PP^d)^n$, determines a codimension $m_i$ linear subspace of the $(d+1)^2$-dimensional affine space $M_{d+1}$.  By the genericity assumption, the condition $g\cdot \widehat{p}\in \widehat{L}_1 \times \cdots \times \widehat{L}_n$ then determines a codimension $\sum_{i=1}^n m_i = d+1$ linear subspace of this space of matrices.  On the other hand, $G\subset M_{d+1}$ is the intersection of a $(d+2)$-dimensional linear subspace of $M_{d+1}$ and the non-linear subvariety $\det = 1$ (see Proposition \ref{prop:matrix}).  The intersection of this latter $(d+2)$-dimensional linear subspace and the former codimension $d+1$ linear subspace is a 1-dimensional subspace of $M_{d+1}$.  There are now two cases: either this line meets the hypersurface $\det = 1$, in which case our sought-after coefficient is 1 and there is a unique $g\in G$ sending $p$ into $L_1\times \cdots \times L_n$, or else this line is contained in the hypersurface $\det = 0$, in which case our coefficient is 0 and there is no matrix in $G$ sending $p$ into $L_1\times \cdots \times L_n$.
\end{proof}

\begin{remark}
In principle one could use this result to directly compute the coefficients of $\delta$, and of any other orbit closure of interest, but we will see later that we actually only need to do this in a special case where it is easier to apply this proposition.  It will follow (Corollary \ref{cor:all1s}) that $\delta$ has coefficients all equal to 1, but we do not need this fact yet.
\end{remark}

\subsection{Component configurations}\label{sec:compconf}

We next discuss a method of associating to each point of the boundary of $T_{d,n}$ a union of special orbit closures whose fundamental cycle (i.e., the sum of the corresponding orbit closure cycles) will turn out to have the same class $\delta$ as the generic orbit closure.  

Fix a closed point 
\[\underline{X} = (H_{v_0} \hookrightarrow X = \cup_{v\in V}X_v,q_1,\ldots,q_n) \in T_{d,n}\] and recall (see \S\ref{cpts}) that each irreducible component $X_v \cong \Bl_{k_v}\PP^d$ is equipped with a parameterized hyperplane $\PP^{d-1} \cong H_v \hookrightarrow X_v$, and that we denote by $\overline{X_v} \cong \PP^d$ the blow-down corresponding to the specified isomorphism and refer to the images of the $k_v \in \mathbb{Z}_{\ge 0}$ exceptional divisors as the blown-up points in $\overline{X_v}$.  Recall also the notion of daughter vertex from Definition \ref{family}.

\begin{definition}
Let $v\in V$ and let $p_1,\ldots,p_{k_v} \in \overline{X_v}$ be the blown-up points with corresponding exceptional divisors $E_1,\ldots,E_{k_v} \subseteq X_v$; then $v$ correspondingly has daughters $v_1,\ldots,v_{k_v}$.  If $v < w$ then we shall call the unique daughter $v_i > v$ satisfying $v_i \le w$ the daughter of $v$ \emph{determined} by $w$.  We also refer to $E_i$ as the exceptional divisor determined by $w$ and $p_i$ as the point determined by $w$.
\end{definition}

Now fix a vertex $v\in V$.  We will construct a divisor $D_v$ on $X$.  If $d=1$ then we simply take $D_v$ to be a smooth point on the $v$-component of $X$; otherwise we proceed as follows.  First choose (the strict transform of) a generic hyperplane in the root component $X_{v_0}$ passing through the point determined by $v$.  The intersection of this hyperplane with the corresponding exceptional divisor is identified in the parent component $X_{v'}$ with a hyperplane inside the hyperplane $H_{v'}$.  Consider the (strict transform of the) hyperplane spanned by this codimension two linear subspace and the point of $X_{v'}$ determined by $v$.  Repeating this process all the way from $v_0$ to $v$ yields a divisor $D_v \subseteq X$ supported in the chain of components from the root component to the component $X_v$.  Note that $D_v$ is a Cartier divisor, and that different choices of initial hyperplane in the root component yield linearly equivalent divisors $D_v$; consequently, the class of the divisor $D_v$ is well-defined by this construction and completely determined by $v$. 
\begin{proposition}\label{prop:configLB}
For any $v\in V$, consider the line bundle $L_v := \mathcal{O}_X(D_v)$.  The following hold:
\begin{enumerate}
\item $L_v$ is basepoint free;
\item $h^0(X,L_v) = d+1$;
\item $h^i(X, L_v) = 0$ for all $i \ge 1$;
\item the restriction of this line bundle satisfies \[L_v|_{X_w} \cong \begin{cases} \mathcal{O}_{X_w}(H_w) &\mbox{if } w = v \\ \mathcal{O}_{X_w} & \mbox{if } w > v\text{ or if }w\text{ is incomparable with }v \\ \mathcal{O}_{X_w}(H_w - E) &\mbox{if } w < v, \text{where }E \subseteq X_w\text{ is the exceptional divisor determined by }v; \end{cases}\] 
\item sections on $X_v$ extend uniquely to $X$: $\Gamma(X_v,L_v|_{X_v}) = \Gamma(X,L_v)$; moreover, the restriction of these sections to any other component $X_w$ yields a complete linear system for $L_v|_{X_w}$.
\end{enumerate}
\end{proposition}

Before proving this, let us state the main geometric consequence.

\begin{corollary}\label{cor:configmap}
The morphism $\varphi_{L_v} : X \rightarrow \PP^d$ induced by $|L_v|$ blows down $X_v$ to $\PP^d$; it contracts each $X_w$ with $w > v$ to the point determined by $w$.  The effect on the remaining components is described inductively: if $w$ is incomparable with $v$ then contract $X_w$ to the point on the parent component determined by $w$; if $w < v$ then project $X_w$ onto the exceptional divisor determined by $v$.
\end{corollary}

\begin{proof}
The corollary follows immediately from the proposition, so let us prove the proposition.  First, note that (1) follows from (4) and (5).  Indeed, to check if $L_v$ has any basepoints it suffices to check this on each component $X_w$, which by (5) is equivalent to checking whether the complete linear system of each restricted line bundle $L_v|_{X_w}$ has any basepoints, and clearly the three line bundles described in (4) have none.  Similarly, (2) follows from (4) and (5) since \[h^0(X,L_v) = h^0(X_v,L_v|_{X_v}) = h^0(X_v, \mathcal{O}_{X_v}(H_v)) = h^0(\PP^d,\mathcal{O}(1)) = d+1.\] Item (4) follows immediately from the construction of the Cartier divisor $D_v$, so it suffices to prove (3) and (5), which we turn to in order.

Consider the short exact sequence \[0 \rightarrow \mathcal{O}_X(-D_v) \rightarrow \mathcal{O}_X \rightarrow \mathcal{O}_{D_v} \rightarrow 0.\]  After tensoring with $L_v = \mathcal{O}_X(D_v)$, we get the long exact cohomology sequence \[\cdots \rightarrow H^{i-1}(\mathcal{O}_{D_v}(D_v)) \rightarrow H^i(\mathcal{O}_X) \rightarrow H^i(L_v) \rightarrow H^i(\mathcal{O}_{D_v}(D_v)) \rightarrow \cdots.\]  The key observation is that $D_v$ is a tree (in fact, a chain, though we won't need this) of $\PP^{d-1}$s and $\mathcal{O}_{D_v}(D_v)$ is the line bundle associated to a Cartier divisor on $D_v$ constructed from a chain of hyperplanes as above.  Therefore, we can assume by induction on the dimension $d$ that $\mathcal{O}_{D_v}(D_v)$ has vanishing higher cohomology.  Thus $H^i(\mathcal{O}_X) \cong H^i(L_v)$ for $i \ge 2$ and $H^1(\mathcal{O}_X) \twoheadrightarrow H^1(L_v)$.  Consequently, to show that $L_v$ has vanishing higher cohomology it suffices to show this for $\mathcal{O}_X$ instead.  

For this purpose, fix $i \ge 1$ and note that for each $w\in V$ we have $H^i(X_w,\mathcal{O}_{X_w}) \cong H^i(\PP^d,\mathcal{O}_{\PP^d}) = 0$, since $X_w$ is a blow-up of $\PP^d$ at a finite number of points (see, e.g., the proof of \cite[Proposition V.3.4]{Hartshorne}).  We can assume, therefore, that $|V| > 1$.  By \cite[Lemma 6.2]{Giansiracusa-Gillam}, there is a short exact sequence \[0 \rightarrow \mathcal{O}_X \rightarrow \mathcal{O}_{X_w}\oplus\mathcal{O}_{\overline{X\setminus X_w}} \rightarrow \mathcal{O}_{X_w\cap \overline{X\setminus X_w}} \rightarrow 0\]  Since $X_w\cap\overline{X\setminus X_w} \cong \PP^{d-1}$ when $w\in V$ is a leaf of the rooted tree, the associated long exact sequence shows that $H^i(\mathcal{O}_X) \cong H^i(\mathcal{O}_{\overline{X\setminus X_w}})$ for all $i \ge 2$ and leaves $w$.  Thus, by induction on $|V|$ and taking $w$ to be a leaf, we deduce that $H^i(\mathcal{O}_X)=0$ for all $i\ge 2$.  Finally, to see that $H^1(\mathcal{O}_X) =0$ we investigate the following portion of this long exact sequence, again with $w$ a leaf: \[H^0(\mathcal{O}_X)\oplus H^0(\mathcal{O}_{\overline{X\setminus X_w}}) \rightarrow H^0(\mathcal{O}_{\PP^{d-1}}) \rightarrow H^1(\mathcal{O}_X) \rightarrow H^1(\mathcal{O}_{\overline{X\setminus X_w}}).\]  The first arrow is surjective, so the middle arrow is the zero map, but then induction on $|V|$ again allows us to assume that $H^1(\mathcal{O}_{\overline{X\setminus X_w}}) = 0$ and hence that this middle arrow is surjective as well.  This proves (3).

For (5), we first note that any hyperplane in $X_v$ avoiding the exceptional divisors meets the hyperplane $H_v \subseteq X_v$ in a codimension two linear subspace $W \subseteq H_v$, and this in turn viewed as a linear subspace $W \subseteq E \subseteq X_{v'}$ of the corresponding exceptional divisor in the parent component extends uniquely to a hyperplane in $X_{v'}$.  Indeed, each point $q \in E$ corresponds to a line through the point lying under $E$, and the hyperplane in $X_{v'}$ is the union of these lines over all $q\in W \subseteq E$.  Proceeding inductively down the chain of components from $X_v$ to the root component $X_{v_0}$, we see that every section in $\Gamma(X_v,L_v|_{X_v})$ extends uniquely to section in $\Gamma(X,L_v)$.  By varying the initial hyperplane in $X_{v}$ in its linear equivalence class we get every codimension two linear subspace contained in $H_v$, and consequently when extending these to hyperplanes in the parent component $X_{v'}$ we get every hyperplane meeting this exceptional divisor, i.e., we obtain the complete linear system $|\mathcal{O}_{X_{v'}}(H_{v'} - E)|$.  Repeating inductively verifies the final claim of (5) for the case $w < v$, and for the remaining cases we need simply observe that any constant function on any other component can be obtained, which is clear. 
\end{proof}

Recall that the data of a stable tree of $\PP^d$s includes a parameterization of each hyperplane $H_v$, i.e., for each $v\in V$ we have a specified map $\PP^{d-1} \hookrightarrow X$ with image $H_v \subseteq X_v$.  We can choose a basis for $|L_v|$ such that the morphism $\varphi_{L_v} : X \rightarrow \PP^d$ sends the $d$ coordinate points and the point $(1:\cdots :1)$ of $H_v \subseteq X_v$ to the corresponding points of the hyperplane $H \subseteq \PP^d$ defined by the vanishing of the first coordinate.  This yields a commutative triangle
\[\xymatrix{X \ar[rr]^{\varphi_{L_v}} & & \PP^d \\ & \PP^{d-1} \ar[ul] \ar[ur] & }\]
which we may view colloquially as the statement that $\varphi_{L_v}$ ``fixes the hyperplane $H_v$ pointwise.''  \emph{In what follows, we shall only work with bases for this complete linear system with this property.}

\begin{definition}\label{def:compconfig}
Fix a closed point \[\underline{X} = (H_{v_0} \hookrightarrow X = \cup_{v\in V}X_v,q_1,\ldots,q_n) \in T_{d,n}.\]  For each vertex $v\in V$, the \emph{$v$-component configuration} is the point configuration \[\pi_v(\underline{X}) := (\varphi_{L_v}(q_1), \ldots, \varphi_{L_v}(q_n)) \in (\PP^d)^n.\]  The \emph{configuration cycle} $Z(\underline{X})$ is the sum of fundamental cycles of $G$-orbit closures of component configurations: \[Z(\underline{X}) := \sum_{v\in V}[\overline{G \pi_v(\underline{X})}] \in \Chow((\PP^d)^n).\]
\end{definition}

\begin{remark}
It follows from Corollary \ref{cor:configmap} and Proposition \ref{prop:fulldim} that all orbits in this sum have full dimension $d+1$.  Moreover, by our assumption above that $\varphi_{L_v}$ fixes $H_v$ pointwise, it follows that each component configuration $\pi_v(\underline{X})$ is well-defined up to the diagonal $G$-action, and hence that the configuration cycle $Z(\underline{X})$ depends only on $\underline{X}\in T_{d,n}$.  
\end{remark}

\begin{figure}[h!]
\includegraphics[scale=0.5]{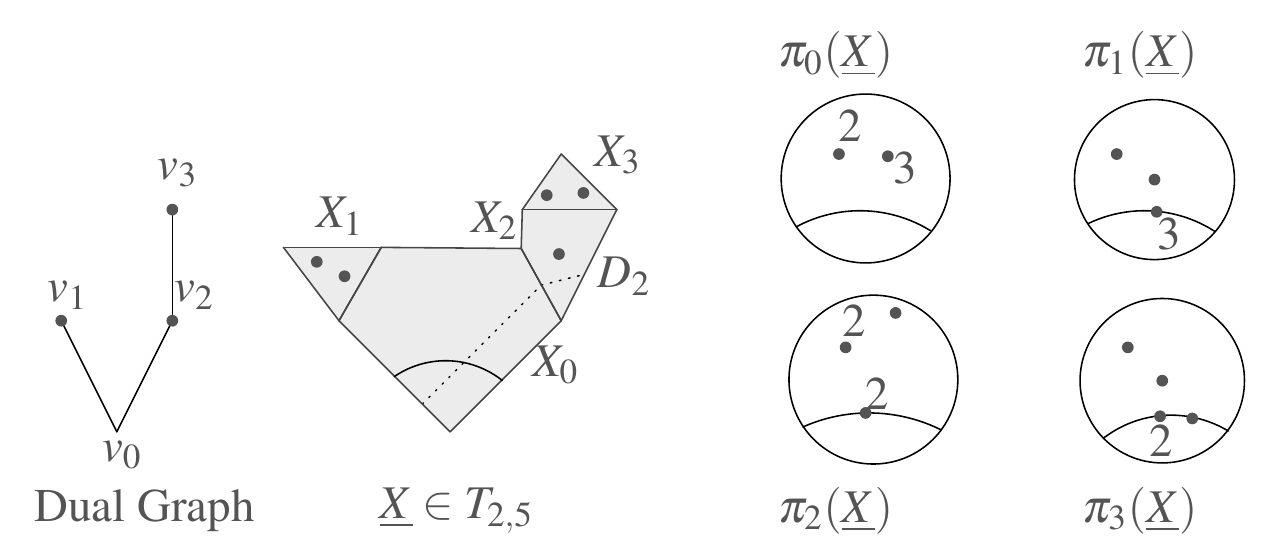}
\caption{Component configurations of a maximal degenerated stable tree in $T_{2,5}$}
\end{figure}


\begin{example}
We illustrate in Figure 3 the component configurations $\pi_v(\underline{X})$ associated to a closed point in $T_{2,5}$.  The numbers indicate the number of points supported at a given location, if it is greater than one.  The dotted line depicts the Cartier divisor $D_2$ defined in the paragraph preceding Proposition \ref{prop:configLB} that is used to construct $\pi_2(\underline{X})$.
\end{example}

To understand degenerations and limits of 1-parameter families in $T_{d,n}$, it will be useful to study the following situation.

\begin{definition}
A closed point $\underline{X}\in T_{d,n}$ is \emph{maximally degenerate} if it lies on a minimal (i.e., deepest) stratum of the boundary stratification, or equivalently, if each blown-down component $\overline{X_v}$ has exactly two special points.  
\end{definition}

A word of caution: unlike the case of $\M_{0,n}$, minimal boundary strata of $T_{d,n}$ have positive dimension when $d > 1$, so maximally degenerate points have moduli here.

\begin{lemma}\label{lemma:maxdegen}
If $\underline{X}\in T_{d,n}$ is maximally degenerate, then for any $v\in V$ the component configuration $\pi_v(\underline{X}) \in (\PP^d)^n$ is supported on $H\cup\{\varphi_{L_v}(p_1),\varphi_{L_v}(p_2)\} \subseteq \PP^d$, where $p_1,p_2\in \overline{X_v}$ are the two special points of this component.
\end{lemma}

\begin{proof}
This follows immediately from Corollary \ref{cor:configmap}.
\end{proof}

\begin{proposition}\label{prop:maxdegenclass}
If $\underline{X}\in T_{d,n}$ is maximally degenerate, then the homology class of $Z(\underline{X})$ has all coefficients equal to 1.
\end{proposition}

\begin{proof}
Note that, by Proposition \ref{prop:homclass}, this is equivalent to proving the following: for any non-negative integers $m_1,\ldots,m_n$ satisfying $\sum m_i = d+1$, and generic linear subspaces $L_i \subseteq \PP^d$ of codimension $m_i$, there is exactly one component $X_v \subseteq X$ such that the $G$-action moves the component configuration $\pi_v(\underline{X}) \in (\PP^d)^n$ into $L := L_1\times \cdots \times L_n$.  Let us now fix these integers $m_i$ and refer to  each as the \emph{weight} of the $i^{\text{th}}$ point in a given configuration.

By Lemma \ref{lemma:maxdegen}, the configuration $\pi_v(\underline{X}) = (p_1,\ldots,p_n)$, $p_i = \varphi_{L_v}(q_i)\in\PP^d$, is supported along the hyperplane $H\subseteq \PP^d$ and at two points off $H$.  Write $\{1,\ldots,n\} = I\sqcup J\sqcup K$ where $\{p_i\}_{i\in I}$ and $\{p_i\}_{i\in J}$ are the two collections of points off $H$, and $\{p_i\}_{i\in K} \subseteq H$.  Set $m_I := \sum_{i\in I} m_i$, and similarly for $J$ and $K$.  Thus $m_I$ and $m_J$ are the total weight of points at the images of the two special points of $\overline{X_v}$, and $m_K$ is the total weight of points lying on $H$.

We claim that the $G$-action moves $\pi_v(\underline{X})$ into $L$ if and only if $m_K = 0$, $m_I > 0$, and $m_J > 0$.  This follows from the observations that $G$ cannot move any $p_i \in H$ into a generic positive codimension subspace, since $G$ fixes $H$ pointwise, and that moving $p_i$ into $L_i$ for all $i\in I$ is equivalent to moving a single point, the support of this collection, into a generic codimension $m_I$ subspace, and similarly for $J$.  Indeed, if $m_K=0$ and, say, $m_I = 0$, then $m_J = d+1$ and a codimension $d+1$ subspace of $\PP^d$ is empty; conversely, if $m_K=0$ and both $m_I$ and $m_J$ are at most $d$, then since the condition that $G$ moves a point off $H$ into a general linear subspace of codimension $m_I$ is a codimension $m_I$ condition on the space of matrices, and this holds independently for $m_J$, one can choose explicit coordinates for generic codimension $m_I$ and codimension $m_J$ linear subspaces and use the fact that $m_I+m_J = d+1 = \dim G$ to see that $G$ can indeed move a pair of points into these two subspaces, and hence the given configuration into $L$.  This verifies the claim.

Consider now the root $v_0$ of the tree associated to $X = \cup X_v$.  The root component configuration $\pi_{v_0}(\underline{X})$ has no points in $H$, so by the claim it can be moved by $G$ into a generic $L = \prod L_i$ if and only if $m_I > 0$ and $m_J >0$, with notation as in the preceding paragraph (since $K = \varnothing$ in this case so $m_K = 0$).  Moreover, no other component yields a configuration that can be moved into $L$ if $m_I > 0$ and $m_J >0$ since these conditions together with Corollary \ref{cor:configmap} imply that such a configuration has positive weight in $H$.  

So suppose that, without loss of generality, $m_J=0$ and $m_I = d+1$.  Let $v'$ be the daughter vertex of $v_0$ determined by the $I$-branch of the tree.  Define $I',J',K'$ as before, but now for this component $X_{v'}$.  Even though $K'$ may be non-empty now, the fact that $m_J = 0$ implies that $m_{K'} = 0$.  Therefore, the claim implies that $G$ moves $\pi_{v'}(\underline{X})$ into $L$ if and only if $m_{I'} >0$ and $m_{J'} > 0$.  Moreover, as before no other component configuration can be moved into $L$ when these conditions hold again since any such configuration has positive weight in $H$.

Repeating this argument inductively, relying on the finite tree structure, completes the proof.  Indeed, at the end of the induction we arrive at a leaf of the tree, so a component with both special points being individual marked points, say $q_1$ and $q_2$; then $m_1 + m_2 = d+1$, since all other $m_i =0$ in this case, and $m_i \le d$ for all $1 \le i \le n$ by definition, so $m_1 > 0$ and $m_2 > 0$, as desired.
\end{proof}

\subsection{Extending the rational map}

By Corollary \ref{cor:BiratCGK}, the CGK-to-Chow map yields a morphism \[T_{d,n} \supseteq T^\circ_{d,n} \xrightarrow{\rho^\circ} (\PP^d)^n\ChowQ G \subseteq \Chow((\PP^d)^n)\] defined on some open subset $T^\circ_{d,n}$ of $T_{d,n}$.  Our goal now is to show that we can take $T^\circ_{d,n} = T_{d,n}$, or in other words, that $\rho^\circ$ extends to a regular morphism $T_{d,n} \rightarrow \Chow((\PP^d)^n)$.  Note that (1) there can be at most one such extension, since $T^\circ_{d,n}$ is dense and the Chow variety is separated, and (2) the image of such an extension is contained in $(\PP^d)^n\ChowQ G$, since the Chow quotient is closed in the Chow variety.  We will therefore denote this extension $\rho : T_{d,n} \rightarrow (\PP^d)^n\ChowQ G$ once we show it exists, and hence will have promoted the CGK-to-Chow map to a CGK-to-Chow morphism.

Consider a $\mathbb{C}$-algebra $R$ that is a discrete valuation ring with maximal ideal $\mathfrak{m}$ and fraction field $K$.   Since $T_{d,n}$ is proper, the valuative criterion implies that any morphism $\psi  : \Spec K \rightarrow T^\circ_{d,n}$ extends to a morphism $\overline{\psi} : \Spec R \rightarrow T_{d,n}$.  Since the Chow variety is also proper, the composition $\psi\rho^\circ$ extends to a morphism $\overline{\psi\rho^\circ} : \Spec R \rightarrow \Chow((\PP^d)^n)$.  We now wish to apply \cite[Theorem 7.3]{Giansiracusa-Gillam}, which says that $\rho^\circ$ extends as desired if and only if for any DVR and any $\psi$ as above, the point $\overline{\psi\rho^\circ}(\mathfrak{m})\in \Chow((\PP^d)^n)$ is uniquely determined by the point $\overline{\psi}(\mathfrak{m}) \in T_{d,n}$.

Even though we have only described the closed points of $T_{d,n}$ in terms of stable trees of $\PP^d$s, and the image of $\psi$ is not a closed point, it follows from general considerations that in fact the set of all field-valued points of $T_{d,n}$ is in bijection with the set of all stable $n$-pointed trees of projective spaces over a field, where a $k$-point corresponds to a tree of $\PP^d_k$s.  Indeed, by \cite[Theorem 3.6.2]{Chen-Gibney-Krashen}, pulling back $T_{d,n}$ and its universal family (cf., Remark \ref{rem:unifam}) along a field extension $\Spec k \rightarrow \Spec \mathbb{C}$ yields $T_{d,n}$ as a $k$-scheme (by which we mean it represents the corresponding functor of screens, see \cite[Definition 3.6.1]{Chen-Gibney-Krashen}) and then $k$-points of this scheme are closed and hence are in bijection with trees of projective spaces over $k$.  Thus, by this  observation, by Remark \ref{rem:unifam}, and by shrinking $T^\circ_{d,n}$ if necessary to be contained in the open stratum of $T_{d,n}$, we are reduced to the following situation.

Consider a flat, proper 1-parameter family of rooted, pointed, stable trees of projective spaces $X_R \rightarrow \Spec R$ where the general fiber $X_K \rightarrow \Spec K$ is smooth with the marked points in general position and the special fiber $X_\mathbb{C} \rightarrow \Spec \mathbb{C}$ is an arbitrary closed point of $T_{d,n}$.  (To ease notation here we briefly omit reference to the root hyperplane and marked points, even though they are part of the data of this family.)  The cycle of the $G$-orbit closure of the point configuration in $X_K \cong \PP^n_K$, namely $\rho^\circ(X_K) \in \Chow((\PP^d)^n)$, has dimension $d+1$ and homology class $\delta$ (cf., \S\ref{subsection:delta}), and it has a unique limit, call it $\lim \rho^\circ(X_K)$, in this Chow variety with respect to the 1-parameter family.  We must show that this limit is uniquely determined by $X_\mathbb{C} \in T_{d,n}(\mathbb{C})$, independent of its smoothing to $X_K$.   We will accomplish this by showing that the limiting cycle is none other than the configuration cycle $Z(X_\mathbb{C})$ from Definition \ref{def:compconfig}.  We first state a lemma that applies here, and more generally without assuming $X_K$ is smooth.

\begin{lemma}\label{lem:DVR}
For an arbitrary $X_K \in T_{d,n}(K)$, we have $Z(X_\mathbb{C}) \subseteq \lim Z(X_K)$ as subvarieties of $(\PP^d)^n$. 
\end{lemma}

\begin{proof}
By the definition of the configuration cycle, it suffices to prove this one component configuration at a time, i.e., that $\pi_v(X_\mathbb{C}) \subseteq \lim Z(X_K)$ for each vertex $v$ in the dual graph of $X_\mathbb{C}$.  Moreover, since the limit cycle is closed and $G$-invariant, it suffices to show that $\varphi_{L_v} : X_\mathbb{C} \rightarrow \PP^d$ sends the $n$-tuple of marked points of $X_\mathbb{C}$ into $\lim Z(X_K) \subseteq (\PP^d)^n$.  We claim that $\varphi_{L_v}$ extends to a morphism $X_R \rightarrow \PP^d_R$ which restricts to the map $\varphi_{L_{w}} : X_K \rightarrow \PP^d_K$ for some vertex $w$ of the dual graph of $X_K$.  We will be done once we verify this claim, since $\varphi_{L_{w}}$ sends the marked points of $X_K$ into $Z(X_K)$ so by continuity of the morphism over $R$ their limits, the marked points of $X_\mathbb{C}$, get sent into the limit of $Z(X_K)$.  So we now turn to the claim about extending $\varphi_{L_v}$ to our 1-parameter family.

Recall that the Cartier divisor $D_v$ inducing the morphism $\varphi_{L_{v}}$ is supported in a chain of components of $X_\mathbb{C}$ starting from the root component and ending with the $v$-component, and the restriction to each component in the chain is a hyperplane. Moreover, recall that all these hyperplanes are determined by the root one and the discrete choice of vertex $v$.  Choose a collection of $d$ smooth points of $X_\mathbb{C}$ which span this $D_v$ hyperplane in the root component.  By viewing these points as sections of $X_\mathbb{C} \rightarrow \Spec\mathbb{C}$ and using the fact that they are smooth, they extend to $d$ smooth sections of the family $X_R \rightarrow \Spec R$.  By taking the span of these sections we get a hyperplane in the root component of $X_R$.  We can then follow the iterative recipe used to construct these Cartier divisors to construct a Cartier divisor $D$ on $X_R$ whose restriction to $X_\mathbb{C}$ is $D_v$ and whose restriction to $X_K$ is of the form $D_w$ for a vertex $w$ of the dual graph of $X_K$ (the corresponding component being one that limits to the $v$-component of $X_\mathbb{C}$).  Thus we have a line bundle $L := \mathcal{O}_{X_R}(D)$ on $X_R$ whose restrictions to $X_K$ and $X_\mathbb{C}$ are $L_w$ and $L_v$, respectively.  By Grauert's Theorem \cite[Corollary III.12.9]{Hartshorne}, using the vanishing of higher cohomology proven in Proposition \ref{prop:configLB}, we see that the global sections of $L_v$ extend to yield the global sections of $L$, so the complete linear system $|L|$ induces a morphism with the desired properties.
\end{proof}

\begin{corollary}\label{cor:all1s}
For any $\underline{X}\in T_{d,n}$, the homology class of the configuration cycle $Z(\underline{X})$ has all coefficients 1.  In particular, this holds for the generic orbit closure class $\delta$.
\end{corollary}

\begin{proof}
By Proposition \ref{prop:homclass}, the class $\delta$ of the configuration cycle of a generic point in the interior of $T_{d,n}$ has all coefficients 0 or 1.  By Lemma \ref{lem:DVR}, the homology class of $Z(\underline{X})$ can only decrease when $\underline{X}$ specializes in $T_{d,n}$.  But by Proposition \ref{prop:maxdegenclass}, the class after maximally degenerating has all coefficients 1, so it must have had these coefficients from the beginning.
\end{proof}

We now return to the situation of interest above.

\begin{proposition}
With notation as above, we have $\lim \rho^\circ(X_K) = Z(X_\mathbb{C})$ in $\Chow((\PP^d)^n)$.
\end{proposition}

\begin{proof}
The containment of supports $Z(X_\mathbb{C}) \subseteq \lim \rho^\circ(X_K)$ proven in Lemma \ref{lem:DVR} implies that it suffices to show that the corresponding homology classes in $H_{2(d+1)}((\PP^d)^n,\mathbb{Z})$ satisfy the inequality $[\lim \rho^\circ(X_K)] \le [Z(X_\mathbb{C})]$, but these are both equal to $\delta$.
\end{proof}

This concludes the proof of the following:
\begin{theorem}
The CGK-to-Chow map is in fact a regular morphism $\rho : T_{d,n} \rightarrow (\PP^d)^n\ChowQ G$.
\end{theorem}

\section{Isomorphism onto the normalization}\label{sec:iso}

In this section we conclude the paper by proving that $T_{d,n}$ is isomorphic to the normalization of the Chow quotient $(\PP^d)^n\ChowQ G$ for any $d \ge 1$, and that for $d=1$ it is isomorphic to the Chow quotient $(\PP^1)^n\ChowQ (\Gm\rtimes\Ga)$ itself.

\subsection{The case $d=1$}

Recall that $T_{1,n} \cong \M_{0,n+1}$.  We can therefore use a trick from \cite{Giansiracusa-Gillam} to reduce the problem of showing the CGK-to-Chow morphism $\rho : T_{1,n} \rightarrow (\PP^1)^n\ChowQ G$ is an isomorphism for all $n$ to the case of $n=3$.  First, since $T_{1,n}$ is proper we know that $\rho$ is surjective, so it suffices to show that the map $T_{1,n} \rightarrow \Chow((\PP^1)^n)$ is an embedding, i.e., an isomorphism onto its image.  For each $I\subseteq\{1,\ldots,n\}$ of size $3$ there is a forgetful map $T_{d,n} \rightarrow T_{d,3}$; this holds for any $d$, by \cite[Remark 3.6.6]{Chen-Gibney-Krashen}, but we only need it for $d=1$ where these maps take the form $\M_{0,n+1} \rightarrow \M_{0,I\cup\{n+1\}}$.  For each such $I$ there is also a morphism $\Chow((\mathbb{P}^1)^n) \rightarrow \Chow((\mathbb{P}^1)^3)$ induced by proper push-forward of cycles \cite[Theorem 6.8]{Kollar-chow} along the projection $(\mathbb{P}^1)^n \rightarrow (\mathbb{P}^1)^3$.  This yields a commutative diagram
\[\xymatrix{T_{1,n} \ar[r] \ar[d] & \Chow((\mathbb{P}^1)^n) \ar[d] \\ \prod_{|I| = 3} T_{1,3} \ar[r] & \prod_{|I|=3}\Chow((\mathbb{P}^1)^3).}\]
Indeed, by separatedness it suffices to check commutativity on the open stratum in $T_{1,n}$, and it holds there since the projection maps $(\mathbb{P}^1)^n \rightarrow (\mathbb{P}^1)^3$ are $G$-equivariant.  We can thereby conclude the proof that $\rho$ is an isomorphism once we demonstrate two key facts: (1) the product of $n=3$ forgetful maps on $T_{1,n}$ is an embedding, and (2) the $d=1,n=3$ case of the CGK-to-Chow map \[\PP^1 \cong T_{1,3} \rightarrow \Chow(\PP^1\times\PP^1\times\PP^1)\] is an embedding.

For the first item, we note that this is equivalent to the following:

\begin{theorem}\label{thm:forgetful}
For any $n \ge 4$, the product of forgetful maps \[\M_{0,n} \rightarrow \prod_{I\in\binom{[n-1]}{3}}\M_{0,I\cup\{n\}} \cong (\PP^1)^{\binom{n-1}{3}}\] is an embedding.
\end{theorem}

Indeed, in the isomorphism $T_{1,n} \cong \M_{0,n+1}$ the last marked point on the RHS is the root hyperplane on the LHS so it is ``remembered'' by all the forgetful maps.  Theorem \ref{thm:forgetful} can be found in the results of \cite{GHvdP}; we  outline another approach here.  The second author and Gillam proved in \cite[Theorem 1.3]{Giansiracusa-Gillam} that the product of all cardinality 4 forgetful maps on $\M_{0,n}$ is an embedding, so the above is a strengthening of this result: to get an embedding, it suffices to use the smaller collection of cardinality 4 forgetful maps which all contain a fixed index.  To prove this, all we need to do is strengthen the key lemma used in the proof of \cite[Theorem 1.3]{Giansiracusa-Gillam}, namely \cite[Lemma 3.1]{Giansiracusa-Gillam}, and then the proof given in \cite{Giansiracusa-Gillam} carries over unchanged to our current setting and yields Theorem \ref{thm:forgetful}.

Let us denote the forgetful map $\M_{0,n} \rightarrow \M_{0,I}$ by $s_I$, since it is also known as a stabilization morphism.  Recall that the boundary divisors of $\M_{0,n}$ are $D_{K,L} \cong \M_{0,|K|+1}\times\M_{0,|L|+1}$ for partitions $[n] = K \sqcup L$ with $|K|,|L| \ge 2$.  

\begin{lemma}
If $x \in \M_{0,n}\setminus D_{K,L}$ then there exists $K' \subseteq K$ and $L'\subseteq L$, each of cardinality two and with $n\in K'\cup L'$, such that $s_{K'\cup L'}(x) \in M_{0,4} \subseteq \M_{0,4}$.
\end{lemma}

\begin{proof}
As noted in the proof of \cite[Lemma 3.1]{Giansiracusa-Gillam} (and suggested there by the anonymous referee), for any subset $S \subset [n]$ there is a unique minimal subtree $\mathcal{T}(S)$ of irreducible components of $x$ whose union contains all marked points of $S$, and the condition that $x$ lies outside $D_{K,L}$ is equivalent to $\mathcal{T}(K) \cap \mathcal{T}(L) \ne \varnothing$.  Without loss of generality, let us assume $n \in L$.  By using the tree structure of the dual graph of $x$, we can find an index $\ell\in L\setminus\{n\}$ and a cardinality two subset $K' \subseteq K$ such that for $L' = \{\ell,n\}$ we have $\mathcal{T}(K')\cap\mathcal{T}(L') \ne \varnothing$.  Indeed, fix a vertex $v\in \mathcal{T}(K)\cap \mathcal{T}(L)$ and let $\ell\in L\setminus\{n\}$ be any marked point such that the unique path in $\mathcal{T}(L)$ from the vertex corresponding to the irreducible component containing $\ell$ to that of $n$ passes through $v$, and let $K' \subseteq K$ be any pair of markings such that the corresponding path in $\mathcal{T}(K)$ passes through $v$.  Since the subtrees associated to $K'$ and $L'$ intersect, the point $s_{K'\cup L'} \in \M_{0,4}$ lies on the boundary divisor $D_{K',L'}$.
\end{proof}

We now turn to the second item mentioned above, namely, we will show that $\rho$ embeds $T_{1,3} \cong \PP^1$ in $\Chow(\PP^1\times\PP^1\times\PP^1)$.  The first observation is that, since the generic $G$-orbit closure in $(\PP^1)^3$ has dimension 2 and homology class with all coefficients 1, the relevant component of this Chow variety is simply the space of hypersurfaces of multi-degree (1,1,1), namely \[\PP H^0((\PP^1)^3,\mathcal{O}(1,1,1)) \cong \PP^5.\]  The proof then proceeds entirely analogously to that of \cite[Lemma 1.6]{Giansiracusa-Gillam}, except in our case it is easier since we already have proven the existence of our morphism from $T_{1,3} \cong \PP^1$ to this $\PP^5$.  Indeed, to see that this morphism is an embedding it suffices to show that it is linear, and to show that it is linear it suffices to check on the interior; but there the map is explicitly given by sending a configuration to a specialization of the cross-ratio functions described in \cite[\S5]{Giansiracusa-Gillam}, where the specialization is simply given by setting the last coordinate equal to $[0:1] \in \PP^1$.  We directly observe in that section of the paper the linear dependence of the cross-ratio on the configuration of 4 points in $\PP^1$, so the specialization used here depends linearly on our 3 points in $\PP^1$.

This completes the proof of our main Theorem stated in the introduction regarding $\M_{0,n}$.

\subsection{The case $d > 1$} 

By Zariski's Main Theorem, a quasi-finite birational morphism to a normal, noetherian scheme is an open immersion.  Since $T_{d,n}$ is normal, our morphism $\rho : T_{d,n} \rightarrow (\PP^d)^n\ChowQ G$ factors through the normalization of the Chow quotient; let us call this induced map $\rho^\nu : T_{d,n} \rightarrow ((\PP^d)^n\ChowQ G)^\nu$.  Since the Chow variety is projective, the Chow quotient is as well, and in particular finite type.  This implies that the normalization is noetherian and the normalization morphism $((\PP^d)^n\ChowQ G)^\nu \rightarrow (\PP^d)^n\ChowQ G$ is finite and birational.  Thus, since $\rho$ is surjective and birational, so is $\rho^\nu$, and moreover if we show that $\rho$ is quasi-finite then $\rho^\nu$ must be as well, and hence by Zariski's Main Theorem we will be able to conclude that $\rho^\nu$ is an isomorphism.  Thus, we are reduced to showing that our CGK-to-Chow morphism $\rho$, or equivalently its composition with the embedding $(\PP^d)^n\ChowQ G \subseteq \Chow((\PP^d)^n)$ which we also denote by $\rho$, is quasi-finite.

We first note that the restriction of $\rho$ to the open dense stratum in $T_{d,n}$ is injective.  Indeed, here the map is of the form $Gp \mapsto \overline{Gp}$ for $p\in(\PP^d)^n$ consisting of $n$ distinct points lying off the hyperplane $H\subset \PP^d$.  Certainly distinct full-dimensional orbits yield distinct orbit closures, hence distinct points of the Chow variety.  Next, we note that no point of the boundary divisor in $T_{d,n}$ can be sent to the same cycle as a point of the open stratum, since the image of the latter is a prime cycle whereas the image of the former is not.  Therefore, we are reduced to showing that the restriction of $\rho$ to the boundary divisor in $T_{d,n}$ is quasi-finite.  Moreover, since this divisor has only finitely many irreducible components, it suffices to show that the restriction of $\rho$ to a single irreducible component of the boundary divisor is quasi-finite.  Recall (see \S\ref{section:props}) that any such component is of the form $T_{d,n-i+1}\times T_{d,i}$ for some $2 \le i \le n-1$.

The general point of the divisor $T_{d,n-i+1}\times T_{d,i}$ corresponds to a rational variety with exactly two components, say $X = X_0 \cup_{E_0=H_1} X_1$, where $X_0 \cong \Bl_1\PP^d$ is the root component and $X_1 \cong \PP^d$ has at least two marked points.  Let us write $[n] = I_0\sqcup I_1$ where $I_0$ indexes the marked points on $X_0$ and $I_1$ indexes those on $X_1$, so $|I_1| = i \ge 2$.  There is a unique blown-up point on $\overline{X_0} = \PP^d$, call it $p$.  The morphism $\rho$ sends any point in the relative interior of this boundary divisor to a sum of two orbit closure cycles, one obtained by placing all $\{q_i\}_{i\in I_1}$ at $p\in \PP^d$, the other obtained by projecting the $\{q_i\}_{i\in I_0}$ onto $E_0 = H_1 \subset X_1 \cong \PP^d$.  Let us call these cycles type 0 and type 1, respectively.  If two points in this relatively open stratum have the same image in the Chow variety, then by definition the sum of their type 0 and type 1 cycles coincide.  But in fact it must be the case that the type 0 cycles coincide and the type 1 cycles coincide, since, for instance, the type 0 cycles are all contained in locus in $(\PP^d)^n$ where the coordinates indexed by $I_1$ coincide whereas the type 1 cycles are never contained in this locus.   Now, for two points in the relative interior of $T_{d,n-i+1}\times T_{d,i}$ to be distinct means that either the $G$-orbits of the corresponding $|I_0|+1$ special points of $\overline{X_0}$ are distinct, or the $G$-orbits of the $|I_1|$ special points of $\overline{X_1}$ are distinct (or both).  But in either case this means that the corresponding orbit closures are distinct, and hence $\rho$ separates these two points.  This shows that $\rho$ is injective on the relative interior of $T_{d,n-i+1}\times T_{d,i}$.  A straightforward iteration of this argument, using the fact that our dual graphs are always trees, applies to all deeper strata and hence shows that $\rho$ is injective on $T_{d,n-i+1}\times T_{d,i}$ itself, as desired.  This completes the proof of our main Theorem stated in the introduction regarding $T_{d,n}$.


\large{\bibliographystyle{amsalpha}}
\bibliography{bib}

\end{document}